\UseRawInputEncoding
\documentclass[12pt,a4paper,UTF8]{article}
\usepackage{mathrsfs}
\usepackage{amsfonts}
\usepackage{amssymb}
\usepackage{amsthm}
\usepackage{hyperref}
\usepackage{cite}
\usepackage[authoryear]{natbib}
\usepackage{xcolor} 
\usepackage{tcolorbox}
\tcbuselibrary{skins, breakable, theorems}
\usepackage{ragged2e}
\usepackage{graphicx}
\usepackage{amsmath, amsbsy}
\usepackage{amsopn, amstext}
\usepackage[ruled,linesnumbered]{algorithm2e}
\usepackage{graphicx}
\usepackage{footnote}
\usepackage{authblk}

\newcommand{\keywords}[1]{\par\noindent\textbf{Keywords:} #1}

\numberwithin{equation}{section}
\newtheorem{thm}{Theorem}[section]
\newtheorem{defi}{Definition}[section]
\newtheorem{lem}{Lemma}[section]
\newtheorem*{remark}{Remark}
\newtheorem{corollary}{Corollary}[section]

\title{\textbf{Convergence of Policy Gradient for Stochastic Linear-Quadratic Control Problem in Infinite Horizon}}

\author[a]{Xinpei Zhang \thanks{Email:zhangxinpei@mail.sdu.edu.cn}}
\author[a]{Guangyan Jia \thanks{Email:jiagy@sdu.edu.cn} \footnote{corresponding author}}
\affil[a]{Zhongtai Securities Institute for Financial Studies, Shandong University, Jinan, 250100, Shandong, China}

\date{}
\begin{document}
	
	\maketitle

	\begin{abstract}
		With the outstanding performance of policy gradient (PG) method in the reinforcement learning field, the convergence theory of it has aroused more and more interest recently. Meanwhile, the significant importance and abundant theoretical researches make the stochastic linear quadratic (SLQ) control problem a starting point for studying PG in model-based learning setting. In this paper, we study the PG method for the SLQ problem in infinite horizon and take a step towards providing rigorous guarantees for gradient methods. Although the cost functional of linear-quadratic problem is typically nonconvex, we still overcome the difficulty based on gradient domination	condition and L-smoothness property, and prove exponential/linear convergence of gradient flow/descent algorithm.
	\end{abstract}
	
	\keywords{Stochastic linear-quadratic control, nonconvex optimization, linear state feedback, policy gradient method, global convergence}
	
	\section{Introduction}
	
		Reinforcement Learning (RL) \citep{2018Reinforcement} recently has an outstanding performance in a wide variety of applications. In particular, RL has been skilful in handling large scale and challenging multistage decision making problem which is uaually computationally intractable through interacting with the environment sequentailly, such as playing Atari \citep{mnih2013playing,mnih2015human}, AlphaGo/AlphaZero \citep{silver2016mastering,silver2017mastering}, autonomous driving \citep{levine2016end} and quantitative finance\citep{nevmyvaka2006reinforcement}. 
		
		With the success of RL, an increasing amount of research has been centered on the theoretical understanding of important issues in RL, including but not limited to algorithm interpretability, computational efficiency, and convergence analysis. On a line of research, \citep{wang2020reinforcement,wang2020continuous} innovatively devise an exploratory formulation of system dynamics with continuous state and control spaces. By means of stochastic calculus, they show that the optimal feedback control distribution for linear-quadratic problem is Gaussian. Their findings interpret why Gaussian exploration is widely adopted in RL and spark a new line of this research direction. 	Further to say, under the exploratory formulation mentioned above, Yanwei Jia and Xun Yu Zhou study policy gradient algorithom\citep{jia2022policygradient}、temporal difference method\citep{jia2022policyevaluation} and q-learning\citep{jia2023q} for reinforcement learning in continuous time and space.
		
		Another recent line of research is to approach the reinforcement learning method from the perspective of policy optimization (PO) which provides a conceptual framework for learning-based control. This idea as a guiding principle in present paper has been stated in \citep{recht2019tour,hu2023toward}. Under the interdisciplinary connection of control theory and RL, the PO formulation of control problem provide a solid foundation for the convergence/complexity theory in RL which catches more and more attention in recent years. In this direction, the linear quadratic regulator, as a foundational and essential class of optimal control problems, has been widely used in testing the effectiveness of RL algorithms. Moreover, the evaluation results are usually without loss of generality, since a wide range of nonlinear problems can be approximated by the linear problem. In present paper, we consider the following stochastic linear quadratic (SLQ) optimal control problem in infinite horizon with the random initial state.
		
		Given a complete probability space $\left(\Omega,\mathcal{F},\mathbb{F},\mathbb{P}\right)$ satisfying the usual condition, on which a standard one-dimensional Brownian motion $\left\{B(t)|t \geq 0\right\}$ is defined. $\mathbb{F} = \left\{ \mathcal{F}_{t} \right\}_{t \geq 0}$ denotes the natural filtration of Brownian motion augmented by all the $\mathbb{P}\text{-null}$ sets in $\mathcal{F}$ and an independent $\sigma$-algebra $\mathcal{F}_0$. Consider the following time-invariant controlled system denoted by $[A,C;B,D]$ for simplicity.:
		\begin{equation*}
			\begin{cases}
				dX(t) = [AX(t)+Bu(t)]ds + [CX(t)+Du(t)]dB(t)  \quad t\geq 0 \\
				X(0) = \xi_0
			\end{cases}
		\end{equation*}
		with quadratic cost functional 
		\begin{equation*}
			J\left(\xi_0;u(\cdot)\right) = E \int_{0}^{+\infty}\Big[ X(t)^TQX(t) + u(t)^TRu(t)\Big]dt.
		\end{equation*}
		Throughout the paper $A, C \in \mathbb{R}^{n \times n}$; $B, D \in \mathbb{R}^{n \times m}$ and $Q \in \mathbb{S}_{+}^{n}$; $ R\in \mathbb{S}_{+}^{m}$ are given constant matrices. The initial state $X(0) = \xi_0$, valued in $\mathbb{R}^n$, is a $\mathcal{F}_0$-measurable random variable.  
		
		The study of the linear quadratic problem has a long history. The seminal works about deterministic model can be traced back to \citep{kalman1960contributions,kalman1960general}, while the pioneering works about the stochastic one is started by \citep{wonham1968matrix}. After nearly a century of thorough study, the interrelated theory has been effectively established and advanced (see, e.g. \citep{Anderson1990Optimal,Davis1977Linear,Bensoussan1982Lectures,Yong1999Stochastic},and the references therein). In the literature, for a SLQ problem in infinite horizon, the optimal control can be witten as linear function of state, and the optimal state feedback gain can be obtained through the solution of the related stochastic algebraic Riccati equation (SARE). In addition, \citep{rami2000linear} use linear matrix inequalities (LMIs) to solve the SARE via a semidefinite programming and its duality.
		
		\textbf{Related work.} As an "end-to-end" approach, the gradient methods are popular and efficient tools for learning feedback control problem. With the successful application of policy gradient algorithm in many fields, an increasing attention has been paid to its convergence theory in recent years. Since the optimal control of the SLQ problem can usually be expressed by linear function of system state which naturally implies parameterization between policies. This significant advantage makes the SLQ problem a starting point of studying PG on model-based learning setting.
		
		In this direction, a great challenge is that the cost functional of linear-quadratic problem is typically nonconvex, both in finite and infinite horizons, deterministic and stochastic dynamics. \citep{fazel2018global} first overcame the difficulty based on gradient domination condition and almost-smoothness condition, and proved global convergence of gradient methods based on infinite horizon LQR problem with random initial state from the perspective of discrete-time. Bu drew similar conclusions on initial-state independent formulation\citep{bu2019lqr} and generalized it to continuous-time analog \citep{bu2020policy}. Furthermore, \citep{mohammadi2019global} transformed the continuous-time linear quadratic regulator problem  into a convex optimization problem through change of variables, and proved exponential/linear convergence of gradient-flow/descent algorithms. In comparison with the former,  \citep{fatkhullin2021optimizing} used the technique which fits for both state and output cases. They established similar results on state feedback, meanwhile, obtained convergence of stationary points for output feedback based on L-smoothness of the objective function. 
		
		All of the aforementioned researches were in the setting of infinite horizon and deterministic dynamics, meanwhile，there also existed some works running over a finite-time horizon\citep{Hambly2021Policy}. In addition,  in the case of linear-quadratic optimal control problem with stochastic dynamics which is characterized by SDEs, \citep{li2022stochastic} solved infinite horizon SLQ problem by policy iteration method which only needs partial information of the system dynamics. \citep{giegrich2024convergence} studied the global linear convergence of PG methods for finite horizon exploratory SLQ problem.
		
		\textbf{Contribution.} Most of the results on the convergence of the gradient method for state feedback were known for the discrete-time case or the continous-time but deterministic case. Inspired by the abovementioned related work, we focus on policy gradient (PG) method for infinite-horizon SLQ problem with rigorous convergence proof in present paper. It is nontrivial to generalize the corresponding proofs and conclusions to the stochastic dynamics setting. The stabilizers of the system $[A,C;B,D]$ can only be characterize by the Lyapunov Equation. Compared with the deterministic case, the lack of the relationship between the feasible set and eigenvalues of the coefficient matrix in SDEs poses novel challenges for researches on the SLQ problem. Meanwhile, we focus on the stochastic system state with controlled diffusion term. It is different from \citep{wang2021global, giegrich2024convergence} in which the diffusion is uncontrolled.

		The rest of this article is organized as follows. In section 2, we focus on the PO formulation of SLQ problem and transform the according PO problem into a matrix optimization problem with equality constraints. In section 3, we discuss the properties of feasible set and objective function. The Gradient domination property and L-smoothness property  provide solid theoretical foundations for the proof of convergence in section 4. The gradient flow described by ordinary differential equation (ODE) convergence exponentially to optimal control for SLQ problem, meanwhile the gradient descent method that arises from the forward Euler discretization of the corresponding ODE convergence linearly to optimal policy. Finally, we provide a numerical example in section 5.   
		
		\textbf{Notation.} Let $\mathbb{R}^k$ denote the k-dimensional Euclidean space; We use $\mathbb{R}^{k \times n }$ to denote the set of all $k \times n $ real matrices; $\mathbb{S}^{k}$ denote the set of symmetric matrices of order k; All the symmetric positive semi-definite (resp., positive definite) matrices are collected by $\overline{\mathbb{S}_{+}^{k}}$ (resp., positive definite $\mathbb{S}_{+}^{k}$); $Tr(\cdot)$ denotes the trace of a square matrix; $\|\cdot\|$ denotes the spectral norm of a matrix; $\|\cdot\|_F$ is its Frobenius norm; $\lambda_1(\cdot)$ (resp., $\lambda_k(\cdot)$) denotes	the minimal (resp., maximal) eigenvalue of a square matrix of order k; $A\succ B$(resp., $A\succeq B)$ means that the matrix A-B is positive (resp., semi-)definite.
		
	\section{Preliminaries and problem statement}
		\subsection{Mean-Square Stabilizability}
			In the literature on stochastic control problem, the problem SLQ in infinite horizon need to consider additional stabilizability conditions which also play an important role in this paper. 
			\begin{defi}
				A feedback control $u(t)=KX(t)$, where $K$ is a constant matrix, is called mean-square stabilizing if for every initial state $X(0)$, the solution of the following equation
				$$
					dX(t)=(A+BK)X(t)dt+(C+DK)X(t)dB(t) 
				$$
				satisfies $\lim_{t\to+\infty}\operatorname{E}[X(t)^{\top}X(t)]=0$. 
			\end{defi}
			\begin{defi}
				The system $[A,C; B,D]$ is called mean-square stablizable if there exists a mean-square stabilizing feedback control of the form $u(t)=KX(t)$ where $K$ is a constant matrix. In this case, $K$ is called a stabilizer of the system $[A,C; B,D]$.The set of all mean-square stabilizers is denoted by $\mathcal{K}:= \mathcal{K}([A,C; B,D])$.
			\end{defi}
			Next, we present the equivalent conditions in verifying the stabilizability that will be used in this paper. From another point, the following theorems also characterize the existence of the stabilizers for system $[A,C;B,D]$, please refer to \citep[Theorem 1]{rami2000linear} or \citep[Lemma 2.2]{sun2018stochastic}. 
			\begin{thm}
				The system $[A,C; B,D]$ is mean-square stabilizable if and only if there is a matrix $K$ such that for any matrix $\Lambda$ there exists a unique solution $P$ to the following matrix equation:
				$$
				(A+BK)^\top P + P(A+BK) + (C+DK)^\top P (C+DK) + \Lambda　= 0.
				$$
				Moreover, if $\Lambda \succ 0$ (respectively, $\Lambda \succeq 0$), then $P \succ 0$ (respectively, $P \succeq 0$). In this case, the matrix $K \in \mathbb{R}^{m \times n}$ is a stabilizer of the system $[A,C; B,D]$. 
			\end{thm}
			\begin{thm}
				The system $[A,C; B,D]$ is mean-square stabilizable if and only if there exist a matrix $Y$ and a symmetric positive definite matrix $P$
			 	\begin{equation*}
					\begin{pmatrix}
							AP+PA^{\top}+BY+Y^{\top}B^{\top} & CP+DY \\
							PC^{\top}+Y^{\top}D^{\top} & -P
						\end{pmatrix} \prec 0
				\end{equation*}
				In this case, the matrix $K=YP^{-1}$ is a stabilizer of the system $[A,C; B,D]$.
			\end{thm}
	
		\subsection{Problem statement}
			In this section, we focus on the PO	formulation of SLQ optimal control problems in an infinite horizon with the following assumptions. 
			
			\begin{itemize}
				\item The system $[A,C; B,D]$ is mean-square stablizable; 
				\item The cost weighting matrices $Q, R$ is positive-definite;
				\item $\Sigma_0 := E X(0)X^{\top}(0)$ is positive-definite.
			\end{itemize}
			
			Based on the study of problem SLQ, the optimal control $u^*(\cdot)$ admits a static linear state feedback representation:
			\begin{equation*}
				u^*(t) = K^*X(t).
			\end{equation*}
			where $K^* = −(R + D^\top PD)^{-1}(B^\top P + D^\top PC)$, the matrix $P \in \mathbb{S}_{+}^{n}$ can be obtained by solving the related SARE. Therefore, for problem SLQ, linear static feedback is a very natural and straightforward candidate form of control, as specified by a constant matrix $K \in \mathbb{R}^{m \times n}$:
			\begin{equation*}
				u(t) = KX(t).
			\end{equation*}
			Then the state dynamics is given by
			\begin{equation*}
				dX(t) = (A+BK)X(t)ds + (C+DK)X(t)dB(t).
			\end{equation*}
			Meanwhile, the cost functional can be rewritten as 
			\begin{equation*}
				J(K) := E \int_{0}^{+\infty} X(t)^T(Q+K^TRK)X(t) dt.
			\end{equation*}
		
			The notation $J(K)$ just emphasizes that the corresponding performance criterion depends only on $K$ when the model is known. It is obvious that this function is well-defined for the set of stabilizing feedback controller, see \citep[Remark 4]{rami2000linear}. Meanwhile, different from the deterministic case\citep{bu2020policy}, we have 
			\begin{equation*}
				\text{Dom}(J(K)) := \left\{ K|J(K) < +\infty \right\} = \mathcal{K}([A,C; B,D]),
			\end{equation*}
			see \citep[Remark 5]{rami2000linear} for details.
		
			Thus, when the controller is linearly parameterized, the above SLQ problem can be formulated as a policy optimization problem of the form
			\begin{equation}
				\min_{K \in \mathcal{K}} J(K).
			\end{equation}
			In addition, suppose $P_K \in \mathbb{S}_{+}^n$ satisfies the following Lyapunov equation:
			\begin{equation}
				(A+BK)^\top P_K + P_K(A+BK) + (C+DK)^\top P_K (C+DK) + Q + K^\top RK　= 0.
			\end{equation}
			It follows that J(K) can be written as:
			\begin{equation*}
				\begin{aligned}
					J(K)= E X(0)^TP_KX(0)=Tr(P_K\Sigma_0),
				\end{aligned}
			\end{equation*}
			see \citep[Lemma 5]{rami2000linear} for the proof. Hence we transform the policy optimization problem (2.1) into a matrix optimization problem with equality constraints :
			\begin{subequations}
				\begin{align}
					\min_{K} J(K)& = Tr(P_K\Sigma_0) \\
					(A+BK)^\top P_K + P_K(A+BK) &+ (C+DK)^\top P_K (C+DK) + Q + K^\top RK　= 0.
				\end{align}
			\end{subequations}
		
	\section{The properties of feasible set and function}
		In order to explore the convergence of gradient methods deeply, we will discuss the properties of feasible set and function in this section.For the policy optimization formulation of SLQ, the intractable analysis difficulties lie in the non-convexity of feasible set. The specific example is given in Appendix B. For a general non-convex optimization problem, gradient method may not converge to the global minimal. However, in SLQ PO problems, several important properties make global convergence possible. The relevant results are stated as follows.
		
		\subsection{The properties of set}
		
			In the learning problems for feedback control synthesis, the topological properties of the set of stabilizers are of utmost important. In a sense, these properties can determine the advantages or limitations of the learning type algorithoms. For example, when the set of stabilizers is disconnected with at least two connected components, gradient-based algorithms may just converge to local minimum point, since local search algorithms usually cannot jump between connected components. At this time, the performance of algorithms depends heavily upon the selection of initial values, which deviates from our desired outcome. Fortunately, the feasible set $\mathcal{K}$ is connected for the LQR PO problem, as we shall see below. 
			
			\begin{lem}
				The feasible set $\mathcal{K}$ is connected.
			\end{lem}
			\begin{proof}[Proof]
				By Theorem 2.2, $K \in \mathcal{K}$ if and only if there exist a matrix $Y$ and a symmetric positive definite matrix $P$
				\begin{equation*}
					\begin{pmatrix}
						AP+PA^{\top}+BY+Y^{\top}B^{\top} & CP+DY \\
						PC^{\top}+Y^{\top}D^{\top} & -P
					\end{pmatrix} \prec 0,
				\end{equation*}
				and $K = YP^{-1}$. 
			
				Denote the set 
				$$
					\mathcal{H}:= \left\{ (Y,P):P\succ0, \begin{pmatrix}
						AP+PA^{\top}+BY+Y^{\top}B^{\top} & CP+DY \\
					PC^{\top}+Y^{\top}D^{\top} & -P
					\end{pmatrix} \prec 0\right\}
				$$ 
				By \citep[Proposition 1.12, Page 50]{dullerud2013course}, the set $\mathcal{H}$ is convex and thus connected. In addition, it is known that $\mathcal{K}$ is the continuous image of $\mathcal{H}$ through the map $(Y,P)\rightarrow YP^{-1}$. Therefore, $\mathcal{K}$ is connected.
			\end{proof}
			
			\begin{lem} 
				For any $K_0 \in \mathcal{K}$, the sublevel set $\mathcal{K}_0 := \left\{K| J(K) \leq J(K_0) \right\}$ is compact.
			\end{lem}
			\begin{proof}[Proof]
				Let $P_{K_i}$ be the solution to the corresponding Lyapunov equation (2.3b) associated with $K_i$; then,
				\begin{equation}
					\begin{aligned}
						J(K_i) &= Tr(P_{K_i}\Sigma_0) = Tr(Y_{K_i}(Q+{K_i}^\top R{K_i})) \\
						&\geq \lambda_1(Y_{K_i}) \lambda_1(R) \|{K_i}\|_F^2 \geq \frac{\lambda_1(\Sigma_0) \lambda_1(R)\|{K_i}\|_F^2}{2(\|A\|+\|B\|\|{K_i}\|_F)}
					\end{aligned}
				\end{equation}
				where $Y_{K_i}$ be the solution of the following dual Lyapunov equation:
				\begin{equation*}
					(A+BK_i)Y_{K_i} + Y_{K_i}(A+BK_i)^{\top} + (C+DK_i)Y_{K_i}(C+DK_i)^T + \Sigma_0 = 0
				\end{equation*}
				The last inequality is based on Lemma A.5.
				
				The compactness property is established by showing that $\mathcal{K}_0$ is bounded and closed. Now suppose that $\mathcal{K}_0$ is unbounded, i.e. $\exists K^{'} \in \mathcal{K}_0$ s.t. $\|K^{'}\|_F \rightarrow +\infty$. Based on Inequality (3.1), $J(K^{'})\rightarrow +\infty$, which is impossible. Therefore, the sublevel set $\mathcal{K}_0$ is bounded. In addition, the continuity of function $J(K)$ implies that the set $\mathcal{K}_0$ contains all of its limit point. Hence, the sublevel set $\mathcal{K}_0$ is compact.
			\end{proof} 
			
			\begin{corollary}
				The function $J(K)$ is coercive i.e. $J(K) \rightarrow +\infty$ if $\|K\|_F \rightarrow +\infty$\cite[Definition 11.10]{bauschke2019convex}.
			\end{corollary}
			
			\begin{corollary}
				For any $K \in \mathcal{K}_0$,
				\begin{equation}
					\begin{aligned}
						\|K\|_F \leq\frac{2\|B\|J(K_0)}{\lambda_1(\Sigma_0)\lambda_1(R)}+\frac{\|A\|}{\|B\|}.
					\end{aligned}
				\end{equation}
			\end{corollary}
			\begin{proof}[Proof]
				Consider furmula (3.1) as a quadratic equation with respect to $\|K\|_F$. 
				\begin{equation*}
					\lambda_1(\Sigma_0)\lambda_1(R)\|K\|_F^2 - 2\|B\|J(K)\|K\|_F - 2 \|A\|J(K) \leq 0
				\end{equation*}
				Bounding its largest root we obtain an explicit expression:
				\begin{equation}
					\begin{aligned}
						\|K\|_F&\leq\frac{\|B\|J(K)+\|B\|J(K)\sqrt{1+\frac{2\|A\|\lambda_1(\Sigma_0)\lambda_1(R)}{\|B\|^2J(K)}}}{\lambda_1(\Sigma_0)\lambda_1(R)} \\
						&\leq\frac{2\|B\|J(K)}{\lambda_1(\Sigma_0)\lambda_1(R)}+\frac{\|A\|}{\|B\|} \leq\frac{2\|B\|J(K_0)}{\lambda_1(\Sigma_0)\lambda_1(R)}+\frac{\|A\|}{\|B\|}.
					\end{aligned}
				\end{equation}
			\end{proof}
		
		\subsection{Gradient expression}
			In order to fully illustrate the gradient method, we need to explicitly write out the gradient expression. Therefore, we start to analyze the differentiability of the function $J(K)$. 
		
			\begin{lem}
				The gradient of $J(K)$ is 
				\begin{equation*}
					\nabla J(K) = 2\Big[RK+B^TP_K+D^TP_K(C+DK)\Big] Y_K
				\end{equation*}
				where $Y_K$ is the solution to the following Lyapunov equation:
				\begin{equation*}
					(A+BK)Y_K + Y_K(A+BK)^T + (C+DK)Y_K(C+DK)^T + \Sigma_0 = 0
				\end{equation*}
				In order to simplify, denote $M := RK+B^TP_K+D^TP_K(C+DK)$. Then the gardient can be expressed by $\nabla J(K) = 2MY_K$.
			\end{lem}
		
			\begin{proof}[Proof]
				Consider the following Lyapunov equation:
				\begin{equation}
					\begin{aligned}
						\Big(A+B(K+\Delta K)\Big)^\top P_{K+\Delta K} + P_{K+\Delta K}\Big(A+B(K+\Delta K)\Big) +  \Big(C+D(K+\Delta K)\Big)^\top \cdot \\
						P_{K+\Delta K} \Big(C+D(K+\Delta K)\Big) + Q + (K+\Delta K)^\top R(K+\Delta K)　= 0.
					\end{aligned}
				\end{equation}
				Subtracting Equation (2.3b) from Equation (3.4). Accurate to the first order in $\Delta K$, the increment of the Lyapunov equation (2.3b) is:  
				\begin{equation*}
					\begin{aligned}
						&(A+BK)^T(P_{K+\Delta K}-P_K) + (P_{K+\Delta K}-P_K)(A+BK) + (C+DK)^T(P_{K+\Delta K}-P_K)(C+ DK)  \\
						&+ {\Delta K}^\top (B^\top P_K + D^\top P_K(C+DK)) + (P_KB+(C+DK)^\top P_KD)\Delta K + {\Delta K}^\top RK + K^\top R\Delta K = 0.
					\end{aligned}
				\end{equation*}
				Then,
				\begin{equation*}
					\begin{aligned}
						J(K+\Delta K) - J(K) &=  Tr\Big((P_{K+\Delta K}-P_K)\Sigma_0\Big) \\
						&= 2Tr\Big[\Big(RK + B^TP_K + D^TP_K(C+DK)\Big) Y_K{\Delta K}^T \Big].
					\end{aligned} 
				\end{equation*}
				Here the second equality operator is based on Lemma A.3, and $Y_K$ is the solution to the following Lyapunov equation:
				\begin{equation*}
					(A+BK)Y_K + Y_K(A+BK)^T + (C+DK)Y_K(C+DK)^T + \Sigma_0 = 0
				\end{equation*}
				Therefore, 
				\begin{equation*}
					\nabla J(K) = 2\Big[RK + B^TP_K + D^TP_K(C+DK)\Big] Y_K
				\end{equation*}
			\end{proof}	
		
		\subsection{L-smoothness property}
			
			In this section, we focus on the L-smoothness property of function $J(K)$, i.e. its second directional derivative is bounded by a constant $L$ which only depends on model parameters and $K_0$. To avoid operating with tensors or vectorization, we shall derive the action of the Hessian by executing Pearlmutter algorithm\citep{Pearlmutter1994Fast} which is widely known in the implementation of poligy gradient algorithms.
			
			\begin{lem}
				For any $E \in \mathbb{R}^{m \times n}$, the action of Hessian $\nabla^2J(K)[E,E]$ is given by
				\begin{equation*}
					\frac{1}{2} \nabla^2J(K)[E,E] = \langle \left(R + D^{\top}P_KD \right)EY_K,E \rangle +2 \Big \langle \left[  B^\top P_K^{'} + D^{\top}P_K^{'}(C+DK) \right]Y_K,E \Big \rangle, 
				\end{equation*}
				where $P^{'}_{K}$ satisfies,
				\begin{equation*}
					(A+BK)^{\top}P^{'}_{K} + P^{'}_{K}(A+BK) + (C+DK)^{\top}P^{'}_{K}(C+DK) + M^{\top}E + E^{\top}M = 0
				\end{equation*}
			\end{lem}
			\begin{proof}[Proof]
				Note that $P^{'}_{K}(K)$ is the differential of the map $ K\rightarrow P_{K}(K)$. It is obviuos that the differential $P^{'}_{K}(K)$ satisfies
				\begin{equation*}
					\begin{aligned}
						(A+BK)^{\top}P^{'}_{K}(K)[E] + P^{'}_{K}(K)[E](A+BK) + (C+DK)^{\top}P^{'}_{K}(K)[E](C+DK) \\
						+ M^{\top}E + E^{\top}M = 0
					\end{aligned}
				\end{equation*}
				In the same way, $Y^{'}_{K}(K)$ satisfies  
				\begin{equation*}
					\begin{aligned}
						(A+BK)Y_K^{'}(K)[E] + Y_K^{'}(K)[E](A+BK)^{\top} + (C+DK)Y_K^{'}(K)[E](C+DK)^{\top} \\
						+ BEY_K + Y_K(BE)^{\top} + DEY_K(C+DK)^{\top} + (C+DK)Y_K(DE)^{\top} = 0
					\end{aligned}
				\end{equation*}
				Denote $P^{'}_{K}:= P^{'}_{K}(K)[E]$; $Y^{'}_{K}:= Y^{'}_{K}(K)[E]$. By the product rule, we have,
				\begin{equation*}
					\frac{1}{2} \nabla^2J(K)[E,E] = \Big \langle \left[RE + B^\top P^{'}_{K} + D^\top P^{'}_K(C+DK) + D^\top P_KDE \right]Y_K , E \Big \rangle   + \langle MY_K^{'} , E \rangle
				\end{equation*}
				Then applying Lemma A.3, we obtain
				\begin{equation*}
					\frac{1}{2} \nabla^2J(K)[E,E] = \langle \left(R + D^{\top}P_KD \right)EY_K,E \rangle +2 \Big \langle \left[  B^\top P_K^{'} + D^{\top}P_K^{'}(C+DK) \right]Y_K,E \Big \rangle 
				\end{equation*}
			\end{proof}
		
			\begin{remark}
				The Hessian matix is positive define at the global minimum $K^{\star}$, since $P^{'}_{K^{\star}} = 0$ and $\nabla^2J(K^{\star})[E,E] = 2 \langle \left(R + D^{\top}P_{K^{\star}}D \right)EY_{K^{\star}},E \rangle$.
			\end{remark}

			\begin{thm}
				The function $J(K)$ is L-smooth on $\mathcal{K}_0$. Specifically, for any $K \in \mathcal{K}_0 $, the following statement holds:
				\begin{equation*}
					|\nabla^2J(K)[E,E]| \leq L\|E\|_F^{2}
				\end{equation*}
				with constant
				\begin{equation*}
					L = \frac{2J(K_0)}{\lambda_1(Q)}\left[ \lambda_n(R) + \frac{J(K_0)\|D\|^2}{\lambda_1(\Sigma_0)} + \left( \|B\| + \|C\|\|D\| + \frac{2\|B\|\|D\|_F^{2}J(K_0)}{\lambda_1(\Sigma_0)\lambda_1(R)} + \frac{\|A\|\|D\|_F^{2}}{\|B\|} \right)\xi \right],
				\end{equation*}
				where
				\begin{equation*}
					\xi = \frac{\sqrt{n}J(K_0)}{\lambda_1(\Sigma_0)}\left(\frac{\tilde{\mu}J(K_0)}{\lambda_1(\Sigma_0)\lambda_1(Q)} + \sqrt{\left(\frac{\tilde{\mu}J(K_0)}{\lambda_1(\Sigma_0)\lambda_1(Q)}\right)^2+\frac{\lambda_n(R)}{\lambda_1(Q)}} \right),
				\end{equation*}
				\begin{equation*}
					\tilde{\mu} = \|B\|+ \|D\|\|C\|_F + \frac{2\|B\|J(K_0)\|D\|^2}{\lambda_1(\Sigma_0)\lambda_1(R)}+\frac{\|A\|\|D\|^2}{\|B\|}. 
				\end{equation*}
			\end{thm}
			The full proof is deferred to Appendix C.1.

		\subsection{Gradient domination property}
			Gradient domination condition, also known as Ležanski-Polyak-{\L}ojasiewicz condition was proposed in \citep{lezanski1963minimumproblem,polyak1963gradient,lojasiewicz1963propriete}. It can serve as a suitable alternative to convexity in non-convex optimization problems \citep{karimi2016linear}.

			\begin{thm}
				Let $K^*$ be an optimal policy. For any $K \in \mathcal{K}_0$, $J(K)$ satisfies the gradient domination condition, i.e.
				\begin{equation*}
					J(K)-J(K^*) \leq \mu \|\nabla J(K)\|_F^2
				\end{equation*}
				where $\mu > 0$ is given by
				\begin{equation*}
					\mu = \frac{4J(K_0)}{\lambda_1(R)\lambda_1(Q)\lambda_1(\Sigma_0)^2}\left( \|A\| + \frac{\|B\|^2J(K_0)}{\lambda_1(R)\lambda_1(\Sigma_0)} \right)^2
				\end{equation*}
			\end{thm}
		
			This theorem can be proved by analyzing the increment of Lyapunov equation (2.3b). The full proof is deferred to Appendix C.2.
			
			\begin{remark}
				Note that if $\nabla J(K) = 0$, then K is an optimal policy. This conclusion is the motivation for choosing a random initial state $X(0)$, i.e. the Positive definiteness of $EX(0)X(0)^{T}$ ensures that all stationary points are globally optimal.
			\end{remark}

	\section{Gradient methods and convergence analysis}
		In this section, we prove the convergence of the gradient methods based on gradient domination property and L-smoothness property.
		
			In continuous case, we consider the following gradient flow characterized by the ordinary differential equation, it can be seen as a continuous-time analog of the gradient descent method.
			\begin{equation}
				\dot{K_t}=-\nabla J(K_t) \qquad  K_0\in\mathcal{K}
			\end{equation}
			\begin{thm}
				For every $K_0 \in \mathcal{K}$, there exists a solution $K_t$ for all $t \geq 0$. Meanwhile we have that:
				\begin{itemize}
					\item[a.] The cost functional $J(K_t)$ is monotone non-increasing. And the trajectory $\left\{K_t\right\}_{t\geq 0}$ can be contained in the sublevel set $\mathcal{K}_0$.
					\item[b.] $\nabla J(K_t) \rightarrow 0$ when $t \rightarrow +\infty$, and the following convergence rate bound holds
					\begin{equation*}
						\min_{0\leq t\leq T}||\nabla J(K_{t})||_F^{2}\leq\frac{J(K_{0})}{T}.
					\end{equation*} 
					\item[c.] The sequence $\{K_t\}_{t \leq 0}$ converges to the global minimum point $K_*$ at a exponential rates:
					\begin{gather}
						J(K_t) - J(K_*) \leq \left(J(K_0)-J(K_*)\right)e^{-\frac{t}{\mu}}, \nonumber\\
						\|K_{t}-K_{*}\|_{F}\leq 4\mu \sqrt{L(J(K_{0})-J(K_{*}))} e^{-\frac{t}{2\mu}}.\nonumber
					\end{gather}
					The value of $\mu$ and L, which corresponds to the representation in Theorem 3.1 and 3.2, are determined by system parameters and initial condition $K_0$.
				\end{itemize} 	
			\end{thm}
		
			In discrete case, we consider the following gradient  descent method：
			\begin{equation}
				K_{n+1} =  K_n-\alpha_n \nabla J(K_n).
			\end{equation}
			where $\alpha_n$ denotes the stepsize of the n-th iteration.
	
			\begin{thm}	 
				Suppose $K_0 \in \mathcal{K}$ is stabiling. For any $0 < \alpha_n \leq \frac{2}{L}$, discrete method generates a nonincreasing sequence:
				\begin{equation}
					J(K_{n+1})\leq J(K_n)-\alpha_n\left(1-\frac{L\alpha_n}{2}\right)||\nabla J(K_n)||_F^2,
				\end{equation}				
				and then we have $K_n \in \mathcal{K}_0$ for all $n$. In addition, if $\varepsilon_1\leqslant\alpha_n\leqslant\frac2L-\varepsilon_2$ for $\varepsilon_1,\varepsilon_2>0$, then $\nabla J(K_n) \rightarrow 0$, and the following convergence rate bound holds with $C = \frac{\varepsilon_1\varepsilon_2L}{2}$:
				\begin{equation*}
					\min_{0\leq n\leq k}||\nabla J(K_{n})||_F^{2}\leq\frac{J(K_{0})}{C(k+1)}.
				\end{equation*} 
				Meanwhile, we have the linear convergence:
				\begin{equation*}
					J(K_n) - J(K_*) \leq q^n\left(J(K_0)-J(K_*)\right),\quad
					\|K_{n}-K_{*}\|_{F}^2\leq cq^n, \quad
					0 \leq q <1.
				\end{equation*}	
			\end{thm}	
		
			The proofs of convergence closely parallel the deterministic analogue, please refer to Theorem 4.1 and 4.2 in \citep{fatkhullin2021optimizing}

	
	\section{Simulation}
	
		In this section, we simulate a numerical example to show the performance of the above gradient methods. First, we need to make a choice of stepsize for the gradient method. In present paper, We use the classical and efficient Barzilai-Borwen method\citep{fletcher2005barzilai,raydan1993barzilai} to estimate the stepsize. 
		
		Let n = 2; m = 1, and set
		\begin{equation*}
			A=\begin{bmatrix}0.3&0.7\\-0.9&0.5\end{bmatrix},B=\begin{bmatrix}0.2\\0\end{bmatrix},C=\begin{bmatrix}0.05&0.03\\0.05&0.02\end{bmatrix},D=\begin{bmatrix}0.05\\0.06\end{bmatrix},\Sigma_0 = \begin{bmatrix}3&0\\0&1\end{bmatrix}.
		\end{equation*}
		The coefficients in cost functional are 
		\begin{equation*}
			Q=\begin{bmatrix}3&0\\0&2\end{bmatrix},\quad R=1.25.
		\end{equation*}
		
		By implementing Algorithm 1, We choose the initial stabilizing controller as $K_0 = [-6,3]$. The value of $\epsilon$ in Algorithm 1 is set to	$10^{-3}$. Finally, Figure 1 illustrate the convergence of gradient descent method and the following relative error is less than $\epsilon$ in 10-15 iterations. 
		\begin{equation*}
			\text{Relative error} = \frac{J(K_n)-J(K^{*})}{J(K_0)-J(K^{*}}
		\end{equation*}
		Finally we obtain 
		$$P^{*}=\begin{bmatrix}61.1422&-35.7578\\-35.7578&81.6610\end{bmatrix},\qquad K_{*}=[-8.3854,4.7642]$$	
		
		\begin{figure}[!h]
			\centering
			\includegraphics[width=0.7\textwidth]{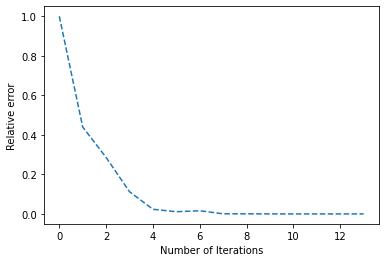}
			\caption{The performance of gradient descent method}
			\label{fig:my_label}
		\end{figure}
		
		\IncMargin{1em}
		\begin{algorithm}
			\SetKwData{Left}{left}\SetKwData{This}{this}\SetKwData{Up}{up}
			\SetKwFunction{Union}{Union}\SetKwFunction{FindCompress}{FindCompress}
			\SetKwInOut{Input}{input}\SetKwInOut{Output}{output}
			
			\Input{$K_0 \in \mathcal{K}$, $\gamma \in (0,1)$, $\epsilon = 10^{-3}$}
			\Output{$K^{*}$ and $P^{*}$}
			\BlankLine
			\emph{Initialization}\;
			\Union{\FindCompress}
			\While{$\| \nabla J(K_n)\|_{F}\geq \epsilon $}{
				Solve Lyapunov Equation: $(A+BK_n)^\top P_{K_n} + P_{K_n}(A+BK_n) + (C+DK_n)^\top P_{K_n} (C+DK_n) + Q + K_n^\top RK_n　= 0$\;
				Solve the dual Lyapunov Equation: $(A+BK_n)Y_{K_n} + Y_{K_n}(A+BK_n)^T + (C+DK_n)Y_{K_n}(C+DK_n)^T + \Sigma_0 = 0$\;
				$\nabla J(K_n) \leftarrow 2\Big[RK_n+B^TP_{K_n}+D^TP_{K_n}(C+DK_n)\Big] Y_{K_n}$\;
				$s_{n-1}\leftarrow Vec(K_n-K_{n-1})$, $y_{n-1} \leftarrow Vec(\nabla J(K_n)-\nabla J(K_{n-1}))$\;
				$\alpha_n \leftarrow {s_{n-1}}^\top s_{n-1} /{s_{n-1}}^\top y_{n-1}$\;
				Gradient step: $K_{n+1} \leftarrow K_n - \alpha_n \nabla J(K_n)$\;
				\If{$J(K_{n+1}) \geq J(K_{n})$}{
					$\alpha_n \leftarrow \gamma\alpha_n$\;
					repeat the Gradient step.
				}
			}
			\caption{Gradient descent method}\label{algo_disjdecomp}
		\end{algorithm}\DecMargin{1em}
	
	\newpage
	
	\appendix
	\section{Some helpful lemmas}
	First, we give two lemmas related to matrix operations.
	\begin{lem} 
		Let $X,Y\in\mathbb{R}^{m\times n}$. Then for any $\alpha>0$,
		\begin{equation*}
			X^\top Y+Y^\top X\preceq\alpha X^\top X+\frac1\alpha Y^\top Y.
		\end{equation*}
	\end{lem}
	
	\begin{lem}
		For all positive semidefinite $X,Y\in\mathbb{R}^{n\times n}$, it holds that
		$$
		\lambda_1(X)Tr(Y) \leq Tr(XY) \leq \lambda_n(X)Tr(Y)
		$$
	\end{lem}
	Then, we state some lemmas in connection with the Lyapunov equation under the assumption that the system $[A,C; B,D]$ is mean-square stabilizable. Denote $A_K = A+BK$; $C_K := C+DK$
	\begin{lem}
		Let $P$ and $Y$ be the solution of the dual Lyapunov equations 
		\begin{equation*}
			\begin{aligned}
				& A_K^\top P + PA_K + C_K^\top PC_K + \Lambda = 0 \\
				& A_KY + YA_K^\top + C_KYC_K^\top + V = 0					 
			\end{aligned}
		\end{equation*}
		Then $Tr(PV) = Tr(Y\Lambda)$.
	\end{lem}
	
	\begin{lem}
		Let $P_1$ and $P_2$ be the solution of the according Lyapunov equations:
		\begin{equation*}
			\begin{aligned}
				& A_K^\top P_1 + P_1A_K + C_K^\top P_1C_K + \Lambda_1 = 0 \\
				& A_K^\top P_2 + P_2A_K + C_K^\top P_2C_K + \Lambda_2 = 0.
			\end{aligned}
		\end{equation*} 
		If $\Lambda_1 \succ \Lambda_2 $, $P_1 \succ P_2$
	\end{lem}
	\begin{lem}
		Consider the Lyapunov matrix equation:
		\begin{equation*}
			A_K^\top P + PA_K + C_K^\top PC_K + \Lambda = 0
		\end{equation*}
		It then follows that 
		\begin{equation*}
			\lambda_1(P) \geq \frac{\lambda_1(\Lambda+C_K^{\top}PC_K)}{2\|A_K\|}
		\end{equation*} 
	\end{lem}
	
	\section{Non-convexity of the set of mean-square stabilizer}
	Let 
	\begin{equation*}
		A = 
		\begin{pmatrix} 
			3 & 0 \\ 
			-5 & 1
		\end{pmatrix}
		B = 
		\begin{pmatrix} 
			-2 & 4 \\ 
			10 & -1
		\end{pmatrix}
		C = 
		\begin{pmatrix} 
			0 & 0 \\ 
			0 & 0
		\end{pmatrix}
		D = 
		\begin{pmatrix} 
			0.00001 & 0.00001 \\ 
			0.00001 & 0.00001
		\end{pmatrix},
	\end{equation*}
	and 
	\begin{equation*}
		K_1 = 
		\begin{pmatrix} 
			-1 & 1 \\ 
			0 & 3
		\end{pmatrix}
		K_2 =
		\begin{pmatrix} 
			-9 & 4 \\ 
			-10 & 5
		\end{pmatrix}.
	\end{equation*}
	The matrix $K_1\in \mathcal{K}$; $K_2\in \mathcal{K}$. But 
	\begin{equation*}
		\frac{1}{2}K_1 + \frac{1}{2}K_2 = 
		\begin{pmatrix} 
			-5 & 2.5 \\ 
			-5 & 4
		\end{pmatrix} \notin \mathcal{K}
	\end{equation*}

	\section{Properties of $J(K)$ }
	
	\subsection{L-smoothness}
	\begin{thm}
		The function $J(K)$ is L-smooth on $\mathcal{K}_0$. Specifically, for any $K \in \mathcal{K}_0 $, the following statement holds:
		\begin{equation*}
			\frac{1}{2}|\nabla^2J(K)[E,E]| \leq L\|E\|_F^{2}
		\end{equation*}
		with constant
		\begin{equation*}
			L = \frac{J(K_0)}{\lambda_1(Q)}\left[ \lambda_n(R) + \frac{J(K_0)\|D\|^2}{\lambda_1(\Sigma_0)} + \left( \|B\| + \|C\|\|D\| + \frac{2\|B\|\|D\|_F^{2}J(K_0)}{\lambda_1(\Sigma_0)\lambda_1(R)} + \frac{\|A\|\|D\|_F^{2}}{\|B\|} \right)\xi \right],
		\end{equation*}
		where
		\begin{equation*}
			\xi = \frac{\sqrt{n}J(K_0)}{\lambda_1(\Sigma_0)}\left(\frac{\tilde{\mu}J(K_0)}{\lambda_1(\Sigma_0)\lambda_1(Q)} + \sqrt{\left(\frac{\tilde{\mu}J(K_0)}{\lambda_1(\Sigma_0)\lambda_1(Q)}\right)^2+\frac{\lambda_n(R)}{\lambda_1(Q)}} \right),
		\end{equation*}
		\begin{equation*}
			\tilde{\mu} = \|B\|+ \|D\|\|C\|_F + \frac{2\|B\|J(K_0)\|D\|^2}{\lambda_1(\Sigma_0)\lambda_1(R)}+\frac{\|A\|\|D\|^2}{\|B\|}. 
		\end{equation*}
	\end{thm}
	\begin{proof}[Proof]
		It follows from Lemma 3.4 that
		\begin{equation*}
			\begin{aligned}
				\frac{1}{2}|\nabla^2J(K)[E,E]| \leq |\langle REY_K,E \rangle| + |\langle D^{\top}P_KDEY_K,E \rangle| + 2|\langle B^\top P_K^{'}Y_K,E \rangle|  \\
				+ 2|\langle D^{\top}P_K^{'}CY_K,E \rangle| + 2|\langle D^{\top}P_K^{'}DKY_K,E \rangle| 
			\end{aligned}
		\end{equation*}
		\begin{itemize}
			\item The first term on the right-hand side of the inequality
		\end{itemize}
		First,
		\begin{equation*}
			\begin{aligned}
				J(K_0) &\geq J(K) = Tr(P_K\Sigma_0) =Tr(Y_K(Q+K^{\top}RK)) \\
				&\geq Tr(Y_K)\lambda_1(Q+K^{\top}RK) \geq \lambda_n(Y_K)\lambda_1(Q+K^{\top}RK) \geq \lambda_n(Y_K)\lambda_1(Q)
			\end{aligned}
		\end{equation*}
		i.e.
		\begin{equation*}
			\lambda_n(Y_K) \leq \frac{J(K_0)}{\lambda_1(Q)}
		\end{equation*}
		Second,
		\begin{equation*}
			\langle REY_K,E \rangle = Tr(REY_KE^{\top}) \leq \lambda_n(R)Tr(EY_KE^{\top}) \leq \lambda_n(R) \lambda_n(Y_K) \|E\|_F^{2} \leq \frac{\lambda_n(R)J(K_0)}{\lambda_1(Q)}\|E\|_F^{2} 
		\end{equation*}
		\begin{itemize}
			\item The second term on the right-hand side of the inequality
		\end{itemize}
		First,
		\begin{equation*}
			J(K_0) \geq J(K) = Tr(P_K\Sigma_0) \geq \lambda_1(\Sigma_0)Tr(P_K) \geq \lambda_1(\Sigma_0)\lambda_n(P_K)
		\end{equation*}
		i.e.
		\begin{equation*}
			\lambda_n(P_K) \leq \frac{J(K_0)}{\lambda_1(\Sigma_0)}
		\end{equation*}
		Second,
		\begin{equation*}
			\begin{aligned}
				\langle D^{\top}P_KDEY_K,E \rangle &= Tr(D^{\top}P_KDEY_KE^\top) = Tr(P_KDEY_KE^\top D^{\top}) \\
				& \leq \lambda_n(P_K)Tr(D^{\top}DEY_KE^\top) \\
				& \leq \lambda_n(P_K)\lambda_n(D^{\top}D)Tr(Y_KE^{\top}E) \\
				& \leq \lambda_n(P_K)\lambda_n(D^{\top}D)\lambda_n(Y_K)Tr(E^{\top}E) \\
				& = \lambda_n(P_K)\lambda_n(Y_K)\|D\|^2\|E\|^2_{F} \\
				& \leq \frac{J(K_0)^2\|D\|^2}{\lambda_1(Q)\lambda_1(\Sigma_0)}\|E\|^2_{F}
			\end{aligned}
		\end{equation*}
		\begin{itemize}
			\item The third term on the right-hand side of the inequality
		\end{itemize}
		First,
		\begin{equation*}
			\begin{aligned}
				\|BEY_K\|_F &= \sqrt{Tr(BEY_KY_KE^{\top}B^{\top})} = \sqrt{Tr(B^{\top}BEY_KY_KE^{\top})} \\
				& \leq \sqrt{\lambda_n(B^{\top}B)Tr(Y_KY_KE^{\top}E)} \leq \sqrt{\lambda_n(B^{\top}B)\lambda_n(Y_KY_K)Tr(E^{\top}E)} \\
				& = \|B\|\cdot\|E\|_F\cdot\lambda_n(Y_K) \leq \frac{J(K_0)\|B\|}{\lambda_1(Q)}\cdot\|E\|_{F}
			\end{aligned}
		\end{equation*}
		Second,
		\begin{equation*}
			\langle B^\top P_K^{'}Y_K,E \rangle = \langle  P_K^{'},BEY_K \rangle \leq \| P_K^{'} \|_F\cdot \|BEY_K \|_F \leq \frac{J(K_0)\|B\|}{\lambda_1(Q)}\cdot\|E\|_{F}\| P_K^{'} \|_F
		\end{equation*}
		\begin{itemize}
			\item The penultimate term on the right-hand side of the inequality
		\end{itemize}
		First,
		\begin{equation*}
			\begin{aligned}
				\|DEY_KC^{\top}\|_F &= \sqrt{Tr(DEY_KC^{\top}CY_KE^{\top}D^{\top})} \\
				& \leq \sqrt{\lambda_n(D^{\top}D)Tr(C^{\top}CY_KE^{\top}EY_K)} \\
				&\leq  \sqrt{\lambda_n(D^{\top}D)\lambda_n(C^{\top}C)Tr(Y_KY_KE^{\top}E)} \\
				&\leq  \sqrt{\lambda_n(D^{\top}D)\lambda_n(C^{\top}C)\lambda_n(Y_KY_K)Tr(E^{\top}E)} \\
				& = \|D\|\cdot\|C\|\cdot\|E\|_F\cdot\lambda_n(Y_K)\\
				&\leq \frac{J(K_0)}{\lambda_1(Q)}\cdot\|D\|\cdot\|C\|\cdot\|E\|_{F}
			\end{aligned}
		\end{equation*}
		Second,
		\begin{equation*}
			\langle D^\top P_K^{'}CY_K,E \rangle = \langle  P_K^{'},DEY_KC^\top \rangle \leq \| P_K^{'} \|_F\cdot \|DEY_KC^\top \|_F \leq \frac{J(K_0)}{\lambda_1(Q)}\cdot\|D\|\|C\|\|E\|_{F}\| P_K^{'} \|_F
		\end{equation*}     
		\begin{itemize}
			\item The last term on the right-hand side of the inequality
		\end{itemize}
		First,
		\begin{equation*}
			\|P_K^{'}D\|_{F} = \sqrt{Tr(P_K^{'}DD^{\top}P_K^{'})} = \sqrt{Tr(P_K^{'}P_K^{'}DD^{\top})} \leq \sqrt{\lambda_n(P_K^{'}P_K^{'})Tr(DD^{\top})} \leq  \| P_K^{'} \|_F \|D\|_F
		\end{equation*}
		Meanwhile,
		\begin{equation*}
			\|DEY_KK^{\top}\|_F \leq \frac{J(K_0)}{\lambda_1(Q)}\cdot\|D\|\cdot\|K \|_F\cdot\|E\|_{F}
		\end{equation*} 
		Second,
		\begin{equation*}
			\langle D^\top P_K^{'}DKY_K,E \rangle  \langle  P_K^{'}D,DEY_KK^\top \rangle \leq \| P_K^{'}D \|_F\cdot \|DEY_KK^\top \|_F \leq  \frac{J(K_0)}{\lambda_1(Q)}\cdot\|D\|_F^2\|K \|_F\|P_K^{'} \|_F \|E\|_{F}
		\end{equation*}
		\begin{itemize}
			\item Consider $\|P_K^{'} \|_F$
		\end{itemize}
		Denote $\alpha\Lambda=  M^{\top}E + E^{\top}M$. Meanwhile,
		\begin{equation*}
			\begin{aligned}
				&M^{\top}E + E^{\top}M \\
				&= \left[ RK+B^{\top}P_K+D^{\top}P_K(C+DK) \right]^{\top}E + E^{\top}\left[ RK+B^{\top}P_K+D^{\top}P_K(C+DK) \right] \\
				&=  K^{\top}RE + E^{\top}RK + P_KBE + (BE)^{\top}P_K + C^{\top}P_KDE  \\
				&\qquad \qquad \qquad \qquad \qquad \qquad \qquad+ (DE)^{\top}P_KC + E^{\top}D^{\top}P_KDK + K^{\top}D^{\top}P_KDE    \\
				& \preceq \alpha K^{\top}RK + \frac{1}{\alpha} E^{\top}RE + \beta_1P_K^2 + \frac{1}{\beta_1} E^{\top}B^{\top}BE + \beta_2C^{\top}P_KP_KC   \\ 
				& \qquad \qquad \qquad \qquad \qquad \qquad + \frac{1}{\beta_2} E^{\top}D^{\top}DE + \beta_3E^{\top}D^{\top}P_KP_KDE + \frac{1}{\beta_3} K^{\top}D^{\top}DK 
			\end{aligned}
		\end{equation*}
		Let $\tilde{P}_{K}^{'}$ satisfies
		\begin{equation*}
			(A+BK)^{\top}\tilde{P}_{K}^{'} + \tilde{P}_{K}^{'}(A+BK) + (C+DK)^{\top}\tilde{P}_{K}^{'}(C+DK) + \alpha \Lambda = 0
		\end{equation*}
		Let's divide the above equation by $\alpha > 0$
		\begin{equation*}
			(A+BK)^{\top}\left(\frac{\tilde{P}_{K}^{'}}{\alpha}\right) + \left(\frac{\tilde{P}_{K}^{'}}{\alpha}\right)(A+BK) + (C+DK)^{\top}\left(\frac{\tilde{P}_{K}^{'}}{\alpha}\right)(C+DK) + \Lambda = 0
		\end{equation*}
		Next, we will discuss how to choose $\alpha, \beta_1, \beta_2$ and $\beta_3$ such that $\frac{\tilde{P}_{K}^{'}}{\alpha} \preceq P_K$.	That is to say, we will focus on an appropriate choice of $\alpha, \beta_1, \beta_2, \beta_3$ guarantees that the inequality $\Lambda \preceq (Q+K^{\top}RK)$ holds. Therefore, Define 
		\begin{equation*}
			\begin{aligned}
				F(\alpha,\beta_1,\beta_2,\beta_3) &= \frac{1}{\alpha} E^{\top}RE + \beta_1P_K^2 + \frac{1}{\beta_1} E^{\top}B^{\top}BE + \beta_2C^{\top}P_KP_KC + \frac{1}{\beta_2} E^{\top}D^{\top}DE \\
				&+ \beta_3E^{\top}D^{\top}P_KP_KDE + \frac{1}{\beta_3} K^{\top}D^{\top}DK  - \alpha Q
			\end{aligned}
		\end{equation*}
		Obviously, 
		\begin{equation*}
			\begin{aligned}
				F(\alpha,\beta_1,\beta_2,\beta_3) &\preceq \frac{1}{\alpha} \lambda_n(E^{\top}RE)I + \beta_1\lambda_n(P_K^2)I + \frac{1}{\beta_1} \lambda_n(E^{\top}B^{\top}BE)I + \beta_2\lambda_n(C^{\top}P_KP_KC)I \\
				&+ \frac{1}{\beta_2} \lambda_n(E^{\top}D^{\top}DE)I + \beta_3\lambda_n(E^{\top}D^{\top}P_KP_KDE)I + \frac{1}{\beta_3}  \lambda_n(K^{\top}D^{\top}DK)I - \alpha \lambda_1(Q)I
			\end{aligned}
		\end{equation*}
		Because, 
		\begin{equation*}
			\begin{aligned}
				&\lambda_n(E^{\top}RE) \leq Tr(E^{\top}RE) \leq \lambda_n(R)\|E\|_F^2\\
				&\lambda_n(E^{\top}B^{\top}BE) \leq Tr(E^{\top}B^{\top}BE) \leq \lambda_n(B^{\top}B)Tr(E^{\top}E) \leq \|B\|^2\|E\|_F^2 \\
				&\lambda_n(C^{\top}P_KP_KC) \leq Tr(C^{\top}P_KP_KC) \leq \lambda_n(P_KP_K)Tr(C^{\top}C) \leq \|P_K\|^2\|C\|_F^2 \\
				&\lambda_n(E^{\top}D^{\top}DE) \leq Tr(E^{\top}D^{\top}DE) \leq \lambda_n(D^{\top}D)Tr(E^{\top}E) \leq \|D\|^2\|E\|_F^2  \\
				&\lambda_n(E^{\top}D^{\top}P_KP_KDE) \leq Tr(E^{\top}D^{\top}P_KP_KDE) \leq \lambda_n(P_KP_K)\lambda_n(D^{\top}D)Tr(E^{\top}E)  \leq \|P_K\|^2\|D\|^2\|E\|_F^2\\
				&\lambda_n(K^{\top}D^{\top}DK) \leq Tr(K^{\top}D^{\top}DK) \leq \lambda_n(D^{\top}D)Tr(K^{\top}K) \leq \|D\|^2\|K\|_F^2  \\
			\end{aligned}
		\end{equation*}
		Consider the matrix
		\begin{equation*}
			\begin{aligned}
				F_1(\alpha,\beta_1,\beta_2,\beta_3) &= \left( \frac{1}{\alpha}\lambda_n(R)\|E\|_F^2 + \beta_1\|P_K\|^2 + \frac{1}{\beta_1}\|B\|^2\|E\|_F^2 + \beta_2 \|P_K\|^2\|C\|_F^2 \right. \\
				&\left.  \frac{1}{\beta_2} \|D\|^2\|E\|_F^2 + \beta_3 \|P_K\|^2\|D\|^2\|E\|_F^2 + \frac{1}{\beta_3} \|D\|^2\|K\|_F^2 - \alpha \lambda_1(Q) \right)I
			\end{aligned}
		\end{equation*}
		Denote $\mu = \|B\|+ \|D\|\|C\|_F + \|D\|^2\|K\|_F$, when we choose 
		\begin{equation*}
			\alpha = \frac{\mu\|P_K\| + \sqrt{\mu^2\|P_K\|^2+\lambda_1(Q)\lambda_n(R)}}{\lambda_1(Q)}\|E\|_F 
		\end{equation*}
		\begin{equation*}
			\beta_1 = \frac{\|B\|\|E\|_F}{\|P_K\|} \qquad \beta_2 = \frac{\|D\|\|E\|_F}{\|P_K\|\|C\|_F} \qquad \beta_3 = \frac{\|K\|_F}{\|P_K\|\|E\|_F}
		\end{equation*}
		we have 
		\begin{equation*}
			F(\alpha, \beta_1, \beta_2, \beta_3) \preceq F_1(\alpha, \beta_1, \beta_2, \beta_3)
		\end{equation*}
		i.e. 
		$$
		P_K^{'} \preceq \tilde{P}_{K}^{'} \preceq \alpha P_K  \preceq \alpha\lambda_n(P_K)I
		$$
		Because
		\begin{equation*}
			\|K\|_F\leq\frac{2\|B\|J(K_0)}{\lambda_1(\Sigma_0)\lambda_1(R)}+\frac{\|A\|}{\|B\|}.
		\end{equation*}
		We have 
		\begin{equation*}
			\mu \leq \|B\|+ \|D\|\|C\|_F + \frac{2\|B\|J(K_0)\|D\|^2}{\lambda_1(\Sigma_0)\lambda_1(R)}+\frac{\|A\|\|D\|^2}{\|B\|} :=\tilde{\mu} 
		\end{equation*}
		Also because
		\begin{equation*}
			\|P_K\| = \lambda_n(P_K) \leq Tr(P_K) \leq \frac{Tr(P_K\Sigma_0)}{\lambda_1(\Sigma_0)} = \frac{J(K)}{\lambda_1(\Sigma_0)} \leq \frac{J(K_0)}{\lambda_1(\Sigma_0)}
		\end{equation*}
		Then, 
		\begin{equation*}
			\alpha \leq \left(\frac{\tilde{\mu}J(K_0)}{\lambda_1(\Sigma_0)\lambda_1(Q)} + \sqrt{\left(\frac{\tilde{\mu}J(K_0)}{\lambda_1(\Sigma_0)\lambda_1(Q)}\right)^2+\frac{\lambda_n(R)}{\lambda_1(Q)}} \right)\|E\|_F 
		\end{equation*}
		Finally, we obtain the bound on the Frobenius norm:
		\begin{equation*}
			\|P_K^{'}\|_F \leq \frac{\sqrt{n}J(K_0)}{\lambda_1(\Sigma_0)}\left(\frac{\tilde{\mu}J(K_0)}{\lambda_1(\Sigma_0)\lambda_1(Q)} + \sqrt{\left(\frac{\tilde{\mu}J(K_0)}{\lambda_1(\Sigma_0)\lambda_1(Q)}\right)^2+\frac{\lambda_n(R)}{\lambda_1(Q)}} \right)\|E\|_F 
		\end{equation*}
		Denote 
		$$
		\xi := \frac{\sqrt{n}J(K_0)}{\lambda_1(\Sigma_0)}\left(\frac{\tilde{\mu}J(K_0)}{\lambda_1(\Sigma_0)\lambda_1(Q)} + \sqrt{\left(\frac{\tilde{\mu}J(K_0)}{\lambda_1(\Sigma_0)\lambda_1(Q)}\right)^2+\frac{\lambda_n(R)}{\lambda_1(Q)}} \right)
		$$
		Therefore,
		\begin{equation*}
			\begin{aligned}
				\frac{1}{2}|\nabla^2J(K)[E,E]| \leq \frac{J(K_0)}{\lambda_1(Q)} &\left[ \lambda_n(R) + \frac{J(K_0)\|D\|^2}{\lambda_1(\Sigma_0)} \right.\\
				&\left. + \left( \|B\| + \|C\|\|D\| + \frac{2\|B\|\|D\|_F^{2}J(K_0)}{\lambda_1(\Sigma_0)\lambda_1(R)} + \frac{\|A\|\|D\|_F^{2}}{\|B\|} \right)\xi \right] \|E\|_F^{2}
			\end{aligned}
		\end{equation*}
		i.e. 
		\begin{equation*}
			|\nabla^2J(K)[E,E]| \leq L\|E\|_F^{2}
		\end{equation*}
		with constant
		\begin{equation*}
			L = \frac{2J(K_0)}{\lambda_1(Q)}\left[ \lambda_n(R) + \frac{J(K_0)\|D\|^2}{\lambda_1(\Sigma_0)} + \left( \|B\| + \|C\|\|D\| + \frac{2\|B\|\|D\|_F^{2}J(K_0)}{\lambda_1(\Sigma_0)\lambda_1(R)} + \frac{\|A\|\|D\|_F^{2}}{\|B\|} \right)\xi \right],
		\end{equation*}
		where
		\begin{equation*}
			\xi = \frac{\sqrt{n}J(K_0)}{\lambda_1(\Sigma_0)}\left(\frac{\tilde{\mu}J(K_0)}{\lambda_1(\Sigma_0)\lambda_1(Q)} + \sqrt{\left(\frac{\tilde{\mu}J(K_0)}{\lambda_1(\Sigma_0)\lambda_1(Q)}\right)^2+\frac{\lambda_n(R)}{\lambda_1(Q)}} \right),
		\end{equation*}
		\begin{equation*}
			\tilde{\mu} = \|B\|+ \|D\|\|C\|_F + \frac{2\|B\|J(K_0)\|D\|^2}{\lambda_1(\Sigma_0)\lambda_1(R)}+\frac{\|A\|\|D\|^2}{\|B\|}. 
		\end{equation*}
	\end{proof}
	
	\subsection{Gradient domination property}	
	\begin{thm}
		Let $K^*$ be an optimal policy. For any $K \in \mathcal{K}_0$, $J(K)$ satisfies the gradient domination condition, i.e.
		\begin{equation*}
			J(K)-J(K^*) \leq \mu \|\nabla J(K)\|_F^2
		\end{equation*}
		where $\mu > 0$ is given by
		\begin{equation*}
			\mu = \frac{4J(K_0)}{\lambda_1(R)\lambda_1(Q)\lambda_1(\Sigma_0)^2}\left( \|A\| + \frac{\|B\|^2J(K_0)}{\lambda_1(R)\lambda_1(\Sigma_0)} \right)^2
		\end{equation*}
	\end{thm}
	
	\begin{proof}[Proof]
		Let $K^*$ be an optimal policy. Let $P_{K^*}$ be the solution of the following Lyapunov equations:
		\begin{equation}
			(A+BK^*)^\top P_K^* + P_K^*(A+BK^*) + (C+DK^*)^\top P_K^* (C+DK^*) + Q + {K^*}^\top RK^*　= 0.
		\end{equation}
		Subtracting the above equation from Equation (2.3b), 
		\begin{equation*}
			\begin{aligned}
				(A+BK^*)^\top (P_K-P_{K^*}) + (P_K-P_{K^*})(A+BK^*) + (C+DK^*)^\top (P_K-P_{K^*})(C+DK^*)  \\
				+ M^\top (K-K^*) + (K-K^*)^\top M - (K-K^*)^\top(R+D^\top P_KD)(K-K^*)=0
			\end{aligned}
		\end{equation*}
		By the Lemma A.1, for any $\alpha>0$, we have
		\begin{equation*}
			M^\top (K-K^*) + (K-K^*)^\top M \preceq \alpha	M^\top M + \frac{1}{\alpha} (K-K^*)^\top (K-K^*)
		\end{equation*}
		Take $\alpha = \frac{1}{\lambda_1(R + D^\top P_KD)}$, then
		\begin{equation*}
			\begin{aligned}
				M^\top (K-K^*) + (K-K^*)^\top M - (K-K^*)^\top(R+D^\top P_KD)(K-K^*) \\ \preceq \frac{1}{\lambda_1(R + D^\top P_KD)} M^\top M 
			\end{aligned}			
		\end{equation*}
		Let X be the solution to
		\begin{equation*}
			A_{K^*}^\top X + XA_{K^*} + C_{K^*}^\top XC_{K^*} + \frac{1}{\lambda_1(R + D^\top P_KD)} M^\top M  = 0
		\end{equation*}
		By the Lemma A.4, $P_K-P_{K^*} \preceq X$. Therefore,
		\begin{equation}
			\begin{aligned}
				J(K)-J(K^*) &= Tr\left[(P_K-P_{K^*})\Sigma_0 \right]  \leq Tr(X\Sigma_0) \\
				& =  \frac{1}{\lambda_1(R + D^\top P_KD)} Tr(Y_{K^*}M^\top M ) \\
				& \leq \frac{\lambda_n(Y_{K^*})}{\lambda_1(R )} Tr(M^\top M ) \\
				& \leq \frac{\lambda_n(Y_{K^*})}{4\lambda_1(R)\lambda_1^2(Y_K)} \|\nabla J(K)\|_F^2
			\end{aligned}
		\end{equation}
		According to Lemma A.5, we have:
		\begin{equation}
			\lambda_1(Y_K) \geq \frac{\lambda_1(\Sigma_0)}{2\|A+BK\|} \geq \frac{\lambda_1(\Sigma_0)}{2(\|A\|+\|B\|\|K\|_F)}.
		\end{equation}
		It follows from corollary 3.2 that
		\begin{equation*}
			\|K\|_F\leq\frac{2\|B\|J(K_0)}{\lambda_1(\Sigma_0)\lambda_1(R)}+\frac{\|A\|}{\|B\|}.
		\end{equation*}
		
		\begin{equation}
			\begin{aligned}
				J(K) &= Tr(P_K\Sigma_0) = Tr(Y_K(Q+K^\top RK)) \\
				&\geq \lambda_1(Y_K) \lambda_1(R) \|K\|_F^2 \geq \frac{\lambda_1(\Sigma_0) \lambda_1(R)\|K\|_F^2}{2(\|A\|+\|B\|\|K\|_F)}
			\end{aligned}
		\end{equation}
		
		\begin{equation}
			\begin{aligned}
				J(K) &= Tr(P_K\Sigma_0) = Tr(Y_K(Q+K^\top RK)) \\
				&\geq \lambda_1(Y_K) \lambda_1(R) \|K\|_F^2 \geq \frac{\lambda_1(\Sigma_0) \lambda_1(R)\|K\|_F^2}{2(\|A\|+\|B\|\|K\|_F)}
			\end{aligned}
		\end{equation}
		In addition, in the proof of the L-smoothness property, we have
		\begin{equation*}
			\lambda_n(Y_{K^*}) \leq \frac{J(K_0)}{\lambda_1(Q)}
		\end{equation*}
		Therefore,
		\begin{equation*}
			J(K)-J(K^*) \leq \frac{4J(K_0)}{\lambda_1(R)\lambda_1(Q)\lambda_1(\Sigma_0)^2}\left( \|A\| + \frac{\|B\|^2J(K_0)}{\lambda_1(R)\lambda_1(\Sigma_0)} \right)^2 \|\nabla J(K)\|_F^2
		\end{equation*}
	\end{proof}

	\bibliographystyle{plainnat}
	\bibliography{ref}

\begin{thebibliography}{42}
\providecommand{\natexlab}[1]{#1}
\providecommand{\url}[1]{\texttt{#1}}
\expandafter\ifx\csname urlstyle\endcsname\relax
  \providecommand{\doi}[1]{doi: #1}\else
  \providecommand{\doi}{doi: \begingroup \urlstyle{rm}\Url}\fi

\bibitem[Anderson and Moore(1990)]{Anderson1990Optimal}
Brian D.~O. Anderson and John B 1941-(John~Barratt) Moore.
\newblock \emph{Optimal control: linear quadratic methods / Brian D. O.
  Anderson, John B. Moore}.
\newblock Prentice Hall, Englewood Cliffs, N.J, 1990.
\newblock ISBN 9780136385608;0136385605;.

\bibitem[Bauschke and Combettes(2019)]{bauschke2019convex}
HH~Bauschke and PL~Combettes.
\newblock Convex analysis and monotone operator theory in hilbert spaces,
  corrected printing, 2019.

\bibitem[Bensoussan(1982)]{Bensoussan1982Lectures}
A.~Bensoussan.
\newblock Lectures on stochastic control.
\newblock In Sanjoy~K. Mitter and Antonio Moro, editors, \emph{Nonlinear
  Filtering and Stochastic Control}, pages 1--62, Berlin, Heidelberg, 1982.
  Springer Berlin Heidelberg.

\bibitem[Bu et~al.(2019)Bu, Mesbahi, Fazel, and Mesbahi]{bu2019lqr}
Jingjing Bu, Afshin Mesbahi, Maryam Fazel, and Mehran Mesbahi.
\newblock Lqr through the lens of first order methods: Discrete-time case.
\newblock \emph{arXiv preprint arXiv:1907.08921}, 2019.

\bibitem[Bu et~al.(2020)Bu, Mesbahi, and Mesbahi]{bu2020policy}
Jingjing Bu, Afshin Mesbahi, and Mehran Mesbahi.
\newblock Policy gradient-based algorithms for continuous-time linear quadratic
  control.
\newblock \emph{arXiv preprint arXiv:2006.09178}, 2020.

\bibitem[Davis(1977)]{Davis1977Linear}
M.~H.~A. Davis.
\newblock \emph{Linear estimation and stochastic control: M. H. A. Davis}.
\newblock Chapman and Hall, London;New York;, 1977.
\newblock ISBN 0470992158;9780470992159;.

\bibitem[Dullerud and Paganini(2013)]{dullerud2013course}
Geir~E Dullerud and Fernando Paganini.
\newblock \emph{A course in robust control theory: a convex approach},
  volume~36.
\newblock Springer Science \& Business Media, 2013.

\bibitem[Fatkhullin and Polyak(2021)]{fatkhullin2021optimizing}
Ilyas Fatkhullin and Boris Polyak.
\newblock Optimizing static linear feedback: Gradient method.
\newblock \emph{SIAM Journal on Control and Optimization}, 59\penalty0
  (5):\penalty0 3887--3911, 2021.

\bibitem[Fazel et~al.(2018)Fazel, Ge, Kakade, and Mesbahi]{fazel2018global}
Maryam Fazel, Rong Ge, Sham Kakade, and Mehran Mesbahi.
\newblock Global convergence of policy gradient methods for the linear
  quadratic regulator.
\newblock In \emph{International conference on machine learning}, pages
  1467--1476. PMLR, 2018.

\bibitem[Fletcher(2005)]{fletcher2005barzilai}
Roger Fletcher.
\newblock On the barzilai-borwein method.
\newblock In \emph{Optimization and control with applications}, pages 235--256.
  Springer, 2005.

\bibitem[Giegrich et~al.(2024)Giegrich, Reisinger, and
  Zhang]{giegrich2024convergence}
Michael Giegrich, Christoph Reisinger, and Yufei Zhang.
\newblock Convergence of policy gradient methods for finite-horizon exploratory
  linear-quadratic control problems.
\newblock \emph{SIAM Journal on Control and Optimization}, 62\penalty0
  (2):\penalty0 1060--1092, 2024.

\bibitem[Hambly et~al.(2021)Hambly, Xu, and Yang]{Hambly2021Policy}
Ben Hambly, Renyuan Xu, and Huining Yang.
\newblock Policy gradient methods for the noisy linear quadratic regulator over
  a finite horizon.
\newblock \emph{SIAM Journal on Control and Optimization}, 59\penalty0
  (5):\penalty0 3359--3391, 2021.
\newblock \doi{10.1137/20M1382386}.
\newblock URL \url{https://doi.org/10.1137/20M1382386}.

\bibitem[Hu et~al.(2023)Hu, Zhang, Li, Mesbahi, Fazel, and
  Ba{\c{s}}ar]{hu2023toward}
Bin Hu, Kaiqing Zhang, Na~Li, Mehran Mesbahi, Maryam Fazel, and Tamer
  Ba{\c{s}}ar.
\newblock Toward a theoretical foundation of policy optimization for learning
  control policies.
\newblock \emph{Annual Review of Control, Robotics, and Autonomous Systems},
  6:\penalty0 123--158, 2023.

\bibitem[Jia and Zhou(2022{\natexlab{a}})]{jia2022policyevaluation}
Yanwei Jia and Xun~Yu Zhou.
\newblock Policy evaluation and temporal-difference learning in continuous time
  and space: A martingale approach.
\newblock \emph{Journal of Machine Learning Research}, 23\penalty0
  (154):\penalty0 1--55, 2022{\natexlab{a}}.

\bibitem[Jia and Zhou(2022{\natexlab{b}})]{jia2022policygradient}
Yanwei Jia and Xun~Yu Zhou.
\newblock Policy gradient and actor-critic learning in continuous time and
  space: Theory and algorithms.
\newblock \emph{Journal of Machine Learning Research}, 23\penalty0
  (275):\penalty0 1--50, 2022{\natexlab{b}}.

\bibitem[Jia and Zhou(2023)]{jia2023q}
Yanwei Jia and Xun~Yu Zhou.
\newblock q-learning in continuous time.
\newblock \emph{Journal of Machine Learning Research}, 24\penalty0
  (161):\penalty0 1--61, 2023.

\bibitem[Kalman(1960)]{kalman1960general}
Rudolf~E Kalman.
\newblock On the general theory of control systems.
\newblock In \emph{Proceedings First International Conference on Automatic
  Control, Moscow, USSR}, pages 481--492, 1960.

\bibitem[Kalman et~al.(1960)]{kalman1960contributions}
Rudolf~Emil Kalman et~al.
\newblock Contributions to the theory of optimal control.
\newblock \emph{Bol. soc. mat. mexicana}, 5\penalty0 (2):\penalty0 102--119,
  1960.

\bibitem[Karimi et~al.(2016)Karimi, Nutini, and Schmidt]{karimi2016linear}
Hamed Karimi, Julie Nutini, and Mark Schmidt.
\newblock Linear convergence of gradient and proximal-gradient methods under
  the polyak-{\l}ojasiewicz condition.
\newblock In \emph{Machine Learning and Knowledge Discovery in Databases:
  European Conference, ECML PKDD 2016, Riva del Garda, Italy, September 19-23,
  2016, Proceedings, Part I 16}, pages 795--811. Springer, 2016.

\bibitem[Levine et~al.(2016)Levine, Finn, Darrell, and Abbeel]{levine2016end}
Sergey Levine, Chelsea Finn, Trevor Darrell, and Pieter Abbeel.
\newblock End-to-end training of deep visuomotor policies.
\newblock \emph{Journal of Machine Learning Research}, 17\penalty0
  (39):\penalty0 1--40, 2016.

\bibitem[Lezanski(1963)]{lezanski1963minimumproblem}
T~Lezanski.
\newblock {\"U}ber das minimumproblem f{\"u}r funktionale in banachschen
  r{\"a}umen.
\newblock \emph{Mathematische Annalen}, 152:\penalty0 271--274, 1963.

\bibitem[Li et~al.(2022)Li, Li, Peng, and Xu]{li2022stochastic}
Na~Li, Xun Li, Jing Peng, and Zuo~Quan Xu.
\newblock Stochastic linear quadratic optimal control problem: a reinforcement
  learning method.
\newblock \emph{IEEE Transactions on Automatic Control}, 67\penalty0
  (9):\penalty0 5009--5016, 2022.

\bibitem[{\L}ojasiewicz(1963)]{lojasiewicz1963propriete}
Stanislaw {\L}ojasiewicz.
\newblock Une propri{\'e}t{\'e} topologique des sous-ensembles analytiques
  r{\'e}els.
\newblock \emph{Les {\'e}quations aux d{\'e}riv{\'e}es partielles},
  117:\penalty0 87--89, 1963.

\bibitem[Mnih et~al.(2013)Mnih, Kavukcuoglu, Silver, Graves, Antonoglou,
  Wierstra, and Riedmiller]{mnih2013playing}
Volodymyr Mnih, Koray Kavukcuoglu, David Silver, Alex Graves, Ioannis
  Antonoglou, Daan Wierstra, and Martin Riedmiller.
\newblock Playing atari with deep reinforcement learning.
\newblock \emph{arXiv preprint arXiv:1312.5602}, 2013.

\bibitem[Mnih et~al.(2015)Mnih, Kavukcuoglu, Silver, Rusu, Veness, Bellemare,
  Graves, Riedmiller, Fidjeland, Ostrovski, et~al.]{mnih2015human}
Volodymyr Mnih, Koray Kavukcuoglu, David Silver, Andrei~A Rusu, Joel Veness,
  Marc~G Bellemare, Alex Graves, Martin Riedmiller, Andreas~K Fidjeland, Georg
  Ostrovski, et~al.
\newblock Human-level control through deep reinforcement learning.
\newblock \emph{nature}, 518\penalty0 (7540):\penalty0 529--533, 2015.

\bibitem[Mohammadi et~al.(2019)Mohammadi, Zare, Soltanolkotabi, and
  Jovanovi{\'c}]{mohammadi2019global}
Hesameddin Mohammadi, Armin Zare, Mahdi Soltanolkotabi, and Mihailo~R
  Jovanovi{\'c}.
\newblock Global exponential convergence of gradient methods over the nonconvex
  landscape of the linear quadratic regulator.
\newblock In \emph{2019 IEEE 58th Conference on Decision and Control (CDC)},
  pages 7474--7479. IEEE, 2019.

\bibitem[Mohammadi et~al.(2021)Mohammadi, Zare, Soltanolkotabi, and
  Jovanovi{\'c}]{mohammadi2021convergence}
Hesameddin Mohammadi, Armin Zare, Mahdi Soltanolkotabi, and Mihailo~R
  Jovanovi{\'c}.
\newblock Convergence and sample complexity of gradient methods for the
  model-free linear--quadratic regulator problem.
\newblock \emph{IEEE Transactions on Automatic Control}, 67\penalty0
  (5):\penalty0 2435--2450, 2021.

\bibitem[Nevmyvaka et~al.(2006)Nevmyvaka, Feng, and
  Kearns]{nevmyvaka2006reinforcement}
Yuriy Nevmyvaka, Yi~Feng, and Michael Kearns.
\newblock Reinforcement learning for optimized trade execution.
\newblock In \emph{Proceedings of the 23rd international conference on Machine
  learning}, pages 673--680, 2006.

\bibitem[Pearlmutter(1994)]{Pearlmutter1994Fast}
Barak~A. Pearlmutter.
\newblock Fast exact multiplication by the hessian.
\newblock \emph{Neural Computation}, 6\penalty0 (1):\penalty0 147--160, 1994.
\newblock \doi{10.1162/neco.1994.6.1.147}.

\bibitem[Polyak(1963)]{polyak1963gradient}
Boris~T Polyak.
\newblock Gradient methods for the minimisation of functionals.
\newblock \emph{USSR Computational Mathematics and Mathematical Physics},
  3\penalty0 (4):\penalty0 864--878, 1963.

\bibitem[Rami and Zhou(2000)]{rami2000linear}
Mustapha~Ait Rami and Xun~Yu Zhou.
\newblock Linear matrix inequalities, riccati equations, and indefinite
  stochastic linear quadratic controls.
\newblock \emph{IEEE Transactions on Automatic Control}, 45\penalty0
  (6):\penalty0 1131--1143, 2000.

\bibitem[Raydan(1993)]{raydan1993barzilai}
Marcos Raydan.
\newblock On the barzilai and borwein choice of steplength for the gradient
  method.
\newblock \emph{IMA Journal of Numerical Analysis}, 13\penalty0 (3):\penalty0
  321--326, 1993.

\bibitem[Recht(2019)]{recht2019tour}
Benjamin Recht.
\newblock A tour of reinforcement learning: The view from continuous control.
\newblock \emph{Annual Review of Control, Robotics, and Autonomous Systems},
  2:\penalty0 253--279, 2019.

\bibitem[Silver et~al.(2016)Silver, Huang, Maddison, Guez, Sifre, Van
  Den~Driessche, Schrittwieser, Antonoglou, Panneershelvam, Lanctot,
  et~al.]{silver2016mastering}
David Silver, Aja Huang, Chris~J Maddison, Arthur Guez, Laurent Sifre, George
  Van Den~Driessche, Julian Schrittwieser, Ioannis Antonoglou, Veda
  Panneershelvam, Marc Lanctot, et~al.
\newblock Mastering the game of go with deep neural networks and tree search.
\newblock \emph{nature}, 529\penalty0 (7587):\penalty0 484--489, 2016.

\bibitem[Silver et~al.(2017)Silver, Schrittwieser, Simonyan, Antonoglou, Huang,
  Guez, Hubert, Baker, Lai, Bolton, et~al.]{silver2017mastering}
David Silver, Julian Schrittwieser, Karen Simonyan, Ioannis Antonoglou, Aja
  Huang, Arthur Guez, Thomas Hubert, Lucas Baker, Matthew Lai, Adrian Bolton,
  et~al.
\newblock Mastering the game of go without human knowledge.
\newblock \emph{nature}, 550\penalty0 (7676):\penalty0 354--359, 2017.

\bibitem[Sun and Yong(2018)]{sun2018stochastic}
Jingrui Sun and Jiongmin Yong.
\newblock Stochastic linear quadratic optimal control problems in infinite
  horizon.
\newblock \emph{Applied Mathematics \& Optimization}, 78\penalty0 (1):\penalty0
  145--183, 2018.

\bibitem[Sutton and Barto(2018)]{2018Reinforcement}
Richard~S Sutton and Andrew~G Barto.
\newblock \emph{Reinforcement Learning:An Introduction}.
\newblock MIT Press, 2018.

\bibitem[Wang and Zhou(2020)]{wang2020continuous}
Haoran Wang and Xun~Yu Zhou.
\newblock Continuous-time mean--variance portfolio selection: A reinforcement
  learning framework.
\newblock \emph{Mathematical Finance}, 30\penalty0 (4):\penalty0 1273--1308,
  2020.

\bibitem[Wang et~al.(2020)Wang, Zariphopoulou, and Zhou]{wang2020reinforcement}
Haoran Wang, Thaleia Zariphopoulou, and Xun~Yu Zhou.
\newblock Reinforcement learning in continuous time and space: A stochastic
  control approach.
\newblock \emph{Journal of Machine Learning Research}, 21\penalty0
  (198):\penalty0 1--34, 2020.

\bibitem[Wang et~al.(2021)Wang, Han, Yang, and Wang]{wang2021global}
Weichen Wang, Jiequn Han, Zhuoran Yang, and Zhaoran Wang.
\newblock Global convergence of policy gradient for linear-quadratic mean-field
  control/game in continuous time.
\newblock In \emph{International Conference on Machine Learning}, pages
  10772--10782. PMLR, 2021.

\bibitem[Wonham(1968)]{wonham1968matrix}
William~M Wonham.
\newblock On a matrix riccati equation of stochastic control.
\newblock \emph{SIAM Journal on Control}, 6\penalty0 (4):\penalty0 681--697,
  1968.

\bibitem[Yong et~al.(1999)Yong, Zhou, and service)]{Yong1999Stochastic}
Jiongmin Yong, Xun~Y. Zhou, and SpringerLink~(Online service).
\newblock \emph{Stochastic Controls: Hamiltonian Systems and HJB Equations},
  volume~43.
\newblock Springer, New York, NY, 1 edition, 1999.

\end{thebibliography}
	
\end{document}